\documentclass[reqno,10pt]{amsart}
\usepackage{amsmath,amssymb}
\usepackage[mathscr]{euscript}
\theoremstyle{plain}
\newtheorem{theorem}{Theorem}[section]

\newtheorem{lemma}[theorem]{Lemma}

\numberwithin{equation}{section}
\numberwithin{theorem}{section}
\newcommand{\mc}[1]{{\mathcal #1}}
\newcommand{\mb}[1]{{\mathbf #1}}
\newcommand{\ms}[1]{{\mathscr #1}}
\newcommand{\bb}[1]{{\mathbb #1}}

\newcommand{\ind}{\mathbf{1}}
\renewcommand{\epsilon}{\varepsilon}
\renewcommand{\tilde}{\widetilde}

\title[
Quantitative ergodicity for the symmetric exclusion process
]
{
Quantitative ergodicity for the symmetric exclusion process with
stationary initial data
}

\author [L.\ Bertini]{Lorenzo Bertini}
\address{Lorenzo Bertini \hfill\break \indent
  Dipartimento di Matematica, Universit\`a di Roma `La Sapienza',
  \hfill\break \indent
  I-00185 Roma, Italy}
\email{bertini@mat.uniroma1.it}

\author [N.\ Cancrini]{Nicoletta Cancrini}
\address{Nicoletta Cancrini\hfill\break \indent
  DIIIE, Universit\`a dell'Aquila,
  \hfill\break \indent
  I-67100, L'Aquila, Italy}
\email{nicoletta.cancrini@univaq.it}

\author [G.\ Posta]{Gustavo Posta}
\address{Gustavo Posta\hfill\break \indent
  Dipartimento di Matematica, Universit\`a di Roma `La Sapienza',
  \hfill\break \indent
  I-00185 Roma, Italy}
\email{gustavo.posta@uniroma1.it}

\begin{document}

\noindent


\begin{abstract}
  We consider the symmetric exclusion process on the $d$-dimensional
  lattice with initial data invariant with respect to space shifts and
  ergodic.
  It is then known that as $t$ diverges the distribution of the
  process at time $t$ converges to a Bernoulli product
  measure. Assuming a summable decay of correlations of the initial
  data, we prove a quantitative version of this convergence by
  obtaining an explicit bound on the Ornstein $\bar d$-distance. The
  proof is based on the analysis of a two species exclusion process
  with annihilation.
\end{abstract}

\maketitle
\thispagestyle{empty}

\section{Introduction}

The analysis of the speed of the convergence to equilibrium for Markov
processes is a major topic in probability theory. Referring to
\cite{lp} for a general overview, we focus the discussion to the case
of reversible stochastic lattice gases, i.e.\ conservative interacting
particles systems satisfying the detailed balance condition with
respect to a Gibbs measure.  If these processes are considered on a
bounded subset $\Lambda$ of the $d$-dimensional lattice they are
ergodic when restricted to the configurations with fixed number of
particles and the corresponding reversible measure is the finite
volume canonical Gibbs measure.  In the high temperature regime, in
\cite{CM,CMR,DP,DP2,LY} it has been shown that both the inverse of the
spectral gap and the logarithmic Sobolev constant grow as the square
of the diameter of $\Lambda$.  On the infinite lattice, stochastic
lattice gases are reversible with respect to the (infinite volume)
canonical Gibbs measures, see \cite[Thm.~2.14]{Ge}.  In the high
temperature regime, by \cite[Thm.~5.14]{Ge}, the extremal elements of
the set of the canonical Gibbs measures consist in the one parameter
family $\{\pi_\rho\}$ where $\pi_\rho$ is the grand-canonical Gibbs
measure with density, i.e.\ expected number of particles per site,
given by $\rho$.
Moreover, as follows from \cite[Thm.~1.72]{Ge}, the semigroup $P_t$,
$t\ge 0$ associated to a reversible stochastic lattice gases is
ergodic in $L_2(\pi_\rho)$ namely,
$\| P_t f - \pi_\rho(f)\|_{L_2(\pi_\rho)}\to 0$ for each
$f\in L_2(\pi_\rho)$.  A quantitative version of this statement can be
obtained when the function $f$ is local, i.e.\ it depends on
the particles configuration only through its value on finitely many
sites.  For this class of functions it has been shown,
for the exclusion and the zero range processes,
that
$\|P_tf - \pi_\rho(f)\|^2_{L_2(\pi_\rho)}\le C\, t^{-d/2}$ for some
constant $C=C(f)$ \cite{BZ,JLQY}.
The case in which the reversible probability is a
grand-canonical Gibbs measure in the high temperature regime is
discussed in \cite{BZ2,CCC,LaY} where a slightly worse bound is
proven.

We here consider the simple symmetric exclusion process. It corresponds to
the infinite temperature case and the probability measure $\pi_\rho$ is
the product Bernoulli measure with parameter $\rho\in [0,1]$.  If the
probability $\mu$ is a suitable local perturbation of
$\pi_\rho$ it has been proven in \cite{BZ} that
$\mathrm{Ent}\big(\mu P_t \big|\pi_\rho\big) \le C t^{-d/2}$ for some
constant $C=C(\mu)$, here $\mathrm{Ent}$ denotes the relative
entropy. See also \cite{Lina} for further details on this issue.  In
general, it appears to be quite difficult to characterize the
probabilities $\mu$ on the configuration space such that $\mu P_t$
converges to $\pi_\rho$ as $t\to\infty$.  However, as proven in
\cite[Thm.~VIII.1.47]{Li}, such convergence holds when $\mu$ is
stationary, i.e.\ invariant with respect to space shifts, 
and ergodic with density $\rho$. Our purpose
is to provide a quantitative version of this statement.  More
precisely, denoting by $\bar d$ the Ornstein distance on the set of
stationary probabilities \cite[\S~I.9.b]{S}, we
prove here that if $\mu$ is stationary, ergodic with
density $\rho$, and has absolutely summable correlations, then
$\bar d (\mu P_t, \pi_\rho) \le C t^{-\gamma(d)}$ for some constant
$C=C(\mu)$ and $\gamma(d)=d/4$ for $d< 4$ and $\gamma(d)=1$ for
$d\ge 4$; see Theorem~\ref{t:1} below.  The proof is achieved by
combining a simple coupling argument with the analysis on the decay of
the density for the two-species symmetric exclusion process with
annihilation \cite{bel1,bel2}, that relies on an analogous result for
the two-species independent random walks \cite{BL}.

Referring to \cite{BL} for more details, we next explain heuristically
the power law decay of the Ornstein $\bar d$ distance.
By \cite[Thm.~I.9.7]{S} the $\bar
d$ distance between $\mu P_t$ and $\pi_\rho$ can be bounded using 
a coupling  between $\mu P_t$ and
$\pi_\rho$ invariant with respect to space shifts:
$\bar d(\mu P_t,\pi_\rho) \le \bb P\big( \eta_0(t) \neq \zeta_0(t)
\big)$.
Here $(\eta_0(t),\zeta_0(t))$ are the occupation
numbers at the origin at time $t\ge 0$ of a two-species 
annihilating exclusion process with equal density. Namely, 
two species of particles that evolve on $\bb Z^d$
according to exclusion processes and annihilating when they meet.
Let $\rho(t)= \bb P\big( \eta_0(t) \neq \zeta_0(t) \big)$ be the
probability that the origin is occupied by either species of
particles. 
In the mean field approximation, $\rho(t)$ decays to zero 
according  to $\dot{\rho}(t) = -\rho(t)^2$ which would imply
$\bar d(\mu P_t,\pi_\rho)\approx t^{-1}$. This
approximation yields the correct behavior when $d\ge 4$ while for $d\le 3$
the Gaussian fluctuations of the initial data become relevant and, 
due to the underlying particle's diffusion, the decay becomes
$\bar d(\mu P_t,\pi_\rho)\approx t^{-d/4}$.

To our knowledge, the present analysis of the symmetric exclusion
process is the first example in which the quantitative ergodicity for
a stochastic lattice gas with stationary initial data has
been achieved. The arguments here developed cover directly the case of
independent random walks. 
We conclude by discussing possible extensions to other models.  As
mentioned before, the crucial ingredient in the proof is the
quantitative decay of the density for the two-species exclusion
process with annihilation.  This decay might be proven for other
attractive stochastic lattice gases such as the zero range process
with increasing rates, see e.g.\ \cite[Thm.~2.5.2]{KL}, or the special
class of reversible stochastic lattice gases in \cite[\S~4.1]{Liln}.
Another simple model for which the quantitative ergodicity could be
investigated is the inclusion process (SIP), indeed this model is self-dual
and a coupling with independent random walks has been constructed in
\cite{OR}.  For the generic case of reversible stochastic lattice
gases where the invariant measure is not a product measure, it seems
however difficult that coupling arguments suffice to establish the
quantitative ergodicity, cfr.\ the corresponding problem of the decay
to equilibrium for local functions in \cite{BZ2,CCC,JLQY,LaY}.  We
remark that another possible setting to discuss the quantitative
ergodicity for stochastic lattice gases with stationary initial data
$\mu$ is the decay rate of the relative entropy per site of
$\mu P_t$ with respect to $\pi_\rho$.
In view of \cite{Ma},
such decay would imply a quantitative decay on the $\bar d$ distance
between $\mu P_t$ and $\pi_\rho$.

\section{Notation and results}
\label{s:nr}

Let $\bb Z^d$ be the $d$-dimensional lattice.  We write
$\Lambda\subset\subset \bb Z^d$ when $\Lambda$ is a finite subset of
$\bb Z^d$.  Set $\Omega:=\{0,1\}^{\bb Z^d}$ that it is considered
endowed with the product topology and the corresponding Borel
$\sigma$--algebra.  Elements of $\bb Z^d$ will be called sites while
elements of $\Omega$ configurations.  For $\eta\in\Omega$ the value of
the configuration $\eta$ at the site $x$, denoted by
$\eta_x\in\{0,1\}$, is interpreted as the absence/presence of a
particle at $x$ and called occupation number. In particular, $\eta_0$
is the occupation number at the origin of $\bb Z^d$.

The simple symmetric exclusion process (SEP) is the Markov process on the
state space $\Omega$ whose generator acts on local functions
$f\colon \Omega\to \bb R$, i.e.\ functions depending on $\eta$ only
through the values $\{\eta_x\}$ for finitely many sites $x$, as
\begin{equation}
  \label{1}
  L f(\eta) = \sum_{\langle x,y\rangle}
  \big[ f\big(\eta^{x,y}\big) - f(\eta)\big].
\end{equation}
The sum is carried out over the unordered edges of $\bb Z^d$
and $\eta^{x,y}$ is the configuration obtained from $\eta$ by
exchanging the occupation numbers at $x$ and $y$,
\begin{equation}
\label{1.2}
  \eta^{x,y}_z :=
  \begin{cases}
    \eta_y & \textrm{if $z=x$,}\\
    \eta_x & \textrm{if $z=y$,}\\
    \eta_z & \textrm{otherwise.}
  \end{cases}
\end{equation}
We denote by $P_t$, $t\ge 0$, the semigroup generated by $L$ that acts
on the Banach space $\ms C(\Omega)$, the family of continuous function
on $\Omega$ endowed with the uniform norm.  We refer to
\cite[Ch.~VIII]{Li} for the construction of this process and its
properties.  In Theorem~1.44 there, it is proven in particular that a
probability $\mu$ on $\Omega$ is invariant for SEP if and only if
$\mu$ is exchangeable, equivalently $\mu$ is a mixture, i.e.\ a
possibly infinite convex combination, of i.i.d.\ Bernoulli measures.
 
Let $\ms P_\tau(\Omega)$ be the set of \emph{stationary} probabilities
on $\Omega$, i.e.\ the probabilities on $\Omega$ invariant with
respect to the space shifts on $\Omega$. Observe that
$\ms P_\tau(\Omega)$ is a convex set and the set of its extremal
points, denoted by $\ms P_{\tau,e}(\Omega)$, consists of the ergodic
probabilities.  For $\rho\in[0,1]$ let
$\pi_\rho\in \ms P_\tau(\Omega)$ the Bernoulli product probability
with parameter $\rho$.  In \cite[Thm.~VIII.1.47 ]{Li} it is proven
that if $\mu\in \ms P_{\tau,e}(\Omega)$ and $\mu(\eta_0)=\rho$ then
$\mu P_t$ weakly converges to $\pi_\rho$ as $t\to +\infty$.  The
purpose of the present analysis is to provide a quantitative version
of this statement with an explicit control on the rate of
convergence. This will be achieved when the probability $\mu$ has
absolutely summable correlations.

To formulate the quantitative ergodicity we need a distance on
$\ms P_{\tau}(\Omega)$. We shall use the so-called Ornstein
$\bar d$ distance. Given $\Lambda\subset\subset \bb Z^d$ let
$d_\Lambda$ be the distance on
$\Omega_\Lambda := \{0,1\}^\Lambda$ defined by
\begin{equation*}
  d_\Lambda (\eta,\zeta) := 
  \sum_{x\in \Lambda}
  |\eta_x -\zeta_x|.
\end{equation*}
Denoting by $\ms P (\Omega_\Lambda)$  the set of probabilities on
$\Omega_\Lambda$, let
$W_\Lambda$ be the $1$-Wasserstein distance on $\ms P (\Omega_\Lambda)$
associated to $d_\Lambda$, i.e.
\begin{equation*}
  W_\Lambda(\mu,\nu) := \inf_{ Q} \int \!Q (d\eta,d\zeta) \,
  d_\Lambda (\eta,\zeta),
\end{equation*}
where the infimum is carried out over all the \emph{couplings} $Q$ of $\mu$
and $\nu$, i.e.\ the set of probabilities on $\Omega\times\Omega$ with
marginals $\mu$ and $\nu$. 
For $\mu\in \ms P_{\tau}(\Omega)$ and $\Lambda\subset \bb Z^d$ let
$\mu_\Lambda$ be the marginal of $\mu$ on $\Omega_\Lambda$.
By a standard super-additive argument, if
$\mu,\nu\in \ms P_{\tau}(\Omega)$ then
\begin{equation}
  \label{2}
  \lim_{\Lambda\uparrow \bb Z^d}
  \frac 1 {|\Lambda|} W_\Lambda(\mu_\Lambda,\nu_\Lambda)
  = \sup_{\Lambda \subset\subset \bb Z^d}
  \frac 1 {|\Lambda|} W_\Lambda(\mu_\Lambda,\nu_\Lambda) 
  = : \bar d(\mu,\nu).
\end{equation}
Moreover, see e.g.\ \cite[Thm.~I.9.7]{S}, $\bar d$ defines a distance on $\ms
P_{\tau}(\Omega)$ that can be represented as
\begin{equation}
  \label{3}
  \bar d(\mu,\nu) = \inf_{Q}
  \int \!Q (d\eta,d\zeta) \, |\eta_0-\zeta_0|
\end{equation}
where we recall that $\eta_0$, $\zeta_0$ are the occupation numbers at
the origin and the infimum is carried out over all the \emph{stationary
  couplings} $Q$ of $\mu$ and $\nu$, i.e.\ the set of couplings
of $\mu$ and $\nu$ that are invariant with respect to space
shifts on $\Omega\times \Omega$. 
By \eqref{2}, the topology induced by $\bar d$ on  $\ms
P_{\tau}(\Omega)$ is finer than the topology induced by the weak
convergence.  
Denoting by $\mathrm{ent}(\mu|\nu)$ the relative entropy \emph{per unit
  of volume} of $\mu$ with respect to $\nu$, we finally mention two
remarkable properties of the Ornstein $\bar d$ distance, see e.g.\
\cite[Thm.~I.9.15~and~I.9.16]{S}:
(i) $\ms P_{\tau,e}(\Omega)$ is $\bar d$ closed,
(ii) for each $\rho\in[0,1]$ the map
$\ms P_{\tau,e}(\Omega)\ni \mu \mapsto \mathrm{ent}(\mu|\pi_\rho)$
is $\bar d$ continuous.

Given two functions $f,g$ and a probability $\mu$ we let
$\mu(f;g) := \mu(fg) -\mu(f)\mu(g)$ be the $\mu$-covariance of $f$ and
$g$. For $\mu\in\ms P_{\tau,e}(\Omega)$ we set
  \begin{equation}
    \label{11}
    A(\mu) := \sum_{x\in \bb Z^d} \big| \mu (\eta_0;\eta_x)\big|.
  \end{equation}
The quantitative ergodicity for SEP with stationary initial data is
then stated as follows.

\begin{theorem}
  \label{t:1}
  For each $d$ there exists a constant $C$ such that
  for any $t> 0$, $\rho\in[0,1]$, and $\mu\in\ms P_{\tau,e}(\Omega)$
  satisfying $\mu(\eta_0)=\rho$,
  \begin{equation}
    \label{12}
    \bar d\big(\mu P_t, \pi_\rho\big) \le C
    \, \frac {\sqrt{A(\mu)}}{t^{\gamma(d)}}
  \end{equation}
  where $\gamma(d) = d/4 $ if $d\le 4$ and $\gamma(d) =1 $ for $d> 4$.
\end{theorem}

The above results implies that if $\mu\in \ms P_{\tau,e}(\Omega)$,
$\mu(\eta_0)=\rho$, and $A(\mu)< +\infty$ then $\mu P_t$ converges to
$\pi_\rho$ in the topology induced by the $\bar d$ distance. In
particular, by remark (ii) above,
$\mathrm{ent}(\mu P_t |\pi_\rho)\to 0$.  It is unclear to us whether
this statement holds without the condition $A(\mu)<+\infty$.

As we next argue, there exist probabilities
$\mu\in\ms P_{\tau,e}(\Omega)$ such that
$\bar d\big(\mu P_t, \pi_\rho\big)$ decays to zero arbitrarily slow as
$t\to +\infty$. Fix two sites $x\neq y$.
By \eqref{2} it is enough to
exhibit $\mu\in\ms P_{\tau,e}(\Omega)$ with $\mu(\eta_0)=\rho$ such that
$\big( \mu P_t\big)\, \big(\eta_x=1,\eta_y=1\big)$ converges to
$\rho^2$ arbitrarily slow. By the so-called self-duality of SEP, see
\cite[Cor.~VIII.1.3]{Li},
\begin{equation*}
  \big( \mu P_t\big) \big(\eta_x=1\,,\,\eta_y=1\big) =
  \mb E_{(x,y)} \mu \big( \eta_{X(t)}=1\,,\, \eta_{Y(t)}=1 \big),
\end{equation*}
where $\big(X(t),Y(t)\big)$, $t\ge 0$, are two particles in exclusion
starting from $(x,y)$. Since correlations of $\mu$ can
decay arbitrarily slow, we deduce that as $t\to +\infty$ 
$\big( \mu P_t\big)\, \big(\eta_x=1,\eta_y=1\big)$ converges to
$\rho^2$ 
arbitrarily slow.

\section{Reduction to the two species SEP with annihilation}
\label{s:2}

In view of \eqref{3}, an upper bound for $\bar d(\mu P_t,\pi_\rho)$
can be obtained by exhibiting a stationary coupling between $\mu P_t$
and $\pi_\rho$.  Starting a time $t=0$ by a stationary coupling of
$\mu$ and $\pi_\rho$ and coupling the corresponding two SEP we obtain,
at time $t>0$, a stationary coupling between $\mu P_t$ and $\pi_\rho$
good enough to produce the bound \eqref{12}.
In words, the coupling between the two SEP can be described as follows.
Particles of the two processes that are at the same site jump 
together while particles alone jump independently.
In formulae, we consider the Markov process whose generator is the
closure of the operator $\tilde L$ that acts on local functions
$F\colon \Omega\times \Omega\to \bb R$ as
\begin{equation}
  \label{2.1}
  \begin{split}
  & \tilde L F \, (\eta,\zeta)  =\sum_{\langle x,y\rangle}
  \Big\{
  \big( 1- \ind_{\eta_x\neq \zeta_x,\eta_y\neq \zeta_y} \big)
  \big[ F(\eta^{x,y},\zeta^{x,y}) - F(\eta,\zeta)\big] \\
  &\qquad
  +\ind_{\eta_x\neq \zeta_x,\eta_y\neq \zeta_y}  
  \big[ F(\eta^{x,y},\zeta)  + F(\eta,\zeta^{x,y}) - 2 F(\eta,\zeta)
  \big] \Big\}.
  \end{split}  
\end{equation}
The corresponding semigroup  is denoted by $\tilde P_t$, $t\ge 0$.
Note that, even if not apparent from the notation, $\tilde L$ is the
operator used in \cite[\S~VIII.2]{Li} to prove the
attractiveness of the exclusion process.  The next statement can be
checked by a direct computation that is omitted.

\begin{lemma}
  \label{t:2.1}
  Let $(\eta(t) ,\zeta(t))$, $t\ge 0$ be the Markov process generated
  by $\tilde L$. Then  
  $\xi(t) := \big\{ \eta_x(t) - \zeta_x(t), \, x\in \bb Z^d \big\}$,
  $t\ge 0$ is a Markov process with state space
  $\mc S = \{-1,0,1\}^{\bb Z^d}$ and generator $\ms L$ that acts on
  the local functions $f \colon \ms S\to \bb R$ as
  \begin{equation}
    \label{2.2}
    \begin{split}
     \ms L f (\xi) = \sum_{\langle x,y\rangle}\Big\{
     \ind_{\xi_x\xi_y\neq -1}
     \big[ f(\xi^{x,y})-f(\xi)\big]
    +2\,  \ind_{\xi_x\xi_y=-1} \big[ f(\xi^{x,y;\dagger}  ) - f(\xi) \big]
    \Big\}
    \end{split}
  \end{equation}
  where $\xi^{x,y}$ has been defined in \eqref{1.2} and 
  \begin{equation*}
    \xi^{x,y;\dagger}_z :=
    \begin{cases}
      0 & \textrm{ if $z\in\{x,y\}$,} \\
      \xi_z & \textrm{ otherwise.}
    \end{cases}  
  \end{equation*}
\end{lemma}

The process generated by \eqref{2.2}, that will be referred to as
the \emph{two species SEP with annihilation}, can be described by
visualizing the sites $x$ where $\xi_x=-1$ as occupied by anti-particles
and the sites $x$ where $\xi_x=1$ occupied by particles.
Particles and anti-particles evolve following two independent SEP and
when a particle jumps over a anti-particle, or conversely a
anti-particle jumps over a particle, the two particles are
annihilated. It can be therefore seen as kinetic model corresponding
to the reaction
$\textrm{anti-particle} + \textrm{particle} \mapsto \emptyset$.  As we
show in the next section, an analysis of this dynamics yields an upper
bound for the Ornstein distance between two SEP with different initial
data.

\section{Long time behavior of the
  two species SEP with annihilation} 
\label{s:3}

In this section we consider the two species SEP with annihilation
obtaining -- for suitable stationary initial data -- an upper bound
for the probability that at time $t>0$ the origin is occupied by
either types of particles.  Given a probability $\wp$ on $\ms S$ the
law of the two species SEP with annihilation, i.e.\ the process
generated by \eqref{2.2}, and initial datum $\wp$ is denoted by
$\bb P_\wp$, the corresponding expectation by $\bb E_\wp$.  For
$\wp\in\ms P_{\tau,e}(\ms S)$, the set of stationary and ergodic
probabilities on $\ms S$, we set
\begin{equation}
    \label{3.1}
    B(\wp) :=
    \sum_{x\in \bb Z^d} \sum_{\alpha,\beta\in\{-1,1\}}
    \big| \wp\big(\xi_0=\alpha \, ; \,\xi_x=\beta\big)
    \big|
\end{equation}
where
$\wp\big(\xi_0=\alpha \, ; \,\xi_x=\beta\big) := \wp\big(\xi_0=\alpha
\, , \,\xi_x=\beta\big) -\wp\big(\xi_0=\alpha)
\wp\big(\xi_x=\beta\big)$.

\begin{theorem}
  \label{t:3.1}
  For each $d$ there exists a constant $C$ such that
  for any $t> 0$ and any $\wp\in\ms P_{\tau,e}(\ms S)$ satisfying
  $\wp(\xi_0=-1)=\wp(\xi_0=1)$
  \begin{equation}
    \label{3.2}
    \bb E_\wp |\xi_0(t)| 
    \le C
    \frac { \sqrt{B(\wp)} }{t^{\gamma (d)}},
  \end{equation}
  where $\gamma (d) = d/4 $ if $d\le 4$ and $\gamma (d) =1 $ for $d> 4$.
\end{theorem}

The analogous statement for two species annihilating independent
random walks and stationary product initial condition has been proven
in \cite{BL}. Relying on the arguments there, the bound
\eqref{3.2} is proven in \cite{bel1,bel2} when the initial datum $\wp$
is a product measure.  This assumption on the initial datum is used
only in \cite[Lemma~2.1]{bel1}; however, as we show in
Lemma~\ref{t:3.3} below, it can be relaxed to the condition
$B(\wp)<+\infty$.  The rest of the arguments in \cite{bel1,bel2}
carries out unchanged to the present setting and yields the statement
in Theorem~\ref{t:3.1}.  Assuming it, we first conclude the proof of
the quantitative ergodicity for the SEP with stationary initial data.

\begin{proof}[Proof of Theorem~\ref{t:1}]
  Given $\mu\in \ms P_{\tau,e}(\Omega)$ with $\mu(\eta_0)=\rho$, let
  $\wp\in \ms P_\tau(\ms S)$ be the law of $\eta -\zeta$ where
  $\eta$ and $\zeta$ are independently sampled from $\mu$ and
  $\pi_\rho$.  By \eqref{11} and \eqref{3.1}, a direct computation
  yields
  \begin{equation*}
    \begin{split}
      B(\wp) & \le  \sum_{\alpha\in \{-1,1\}}
      [1-\rho(1-\rho)]
      \, \big| \mu\big(\eta_0 = \tfrac{1+\alpha}2 \,
         ;\,  \eta_0 = \tfrac{1+\alpha}2\big)\big|
         \\
         & \;\; + \sum_{\alpha\neq \beta \in \{-1,1\}}
    \rho(1-\rho) \, \big| \mu\big(\eta_0 = \tfrac{1+\alpha}2 \,
         ;\,  \eta_0 = \tfrac{1+\beta}2\big)\big|
    \\
   &  \;\; + \sum_{x\neq 0} \sum_{\alpha,\beta\in \{-1,1\}}
    \pi_\rho\big(\zeta_0 = \tfrac{1-\alpha}2 \,
    ,\,  \zeta_x = \tfrac{1-\beta}2\big)
    \, \big|\mu\big(\eta_0 = \tfrac{1+\alpha}2 \,
    ;\,  \eta_x = \tfrac{1+\beta}2\big)\big| 
    \\
    & \le A(\mu).
  \end{split}
\end{equation*}
Since the process $\big(\eta (t),\zeta(t)\big)$, $t\ge 0$ couples two SEP, the
probability $(\mu\otimes \pi_\rho) \tilde P_t$ is a coupling of
$\mu P_t$ and $\pi_\rho P_t=\pi_\rho$.  Here
$P_t$ and $\tilde P_t$, $t\ge 0$, are the semigroups associated to 
SEP and \eqref{2.1}, respectively.
Moreover, as the probability $\mu\otimes \pi_\rho$ on $\Omega\times
\Omega$ is invariant with respect to space shifts, 
$(\mu\otimes \pi_\rho) \tilde P_t$ is a stationary coupling of
$\mu P_t$ and $\pi_\rho$. 
According to Lemma~\ref{t:2.1}, $\xi(t) =\eta(t)-\zeta(t)$ is
distributed as the two-species SEP with annihilation, i.e.\ 
the process generated by \eqref{2.2}, whose law is denoted by $\bb
E_\wp$.
Hence, by \eqref{3} and Theorem~\ref{t:3.1},
\begin{equation*}
  \bar d \big( \mu P_t, \pi_\rho) \le  \bb E_\wp |\xi_0(t)| \le C
  \frac {\sqrt{A(\mu)}}{t^{\gamma(d)}}
\end{equation*}
for some constant $C$ depending only on $d$.
\end{proof}

In order to extend the result in \cite{bel1,bel2} to non-product
initial data, we need to realize the two species SEP with annihilation
on the probability space associated to the so-called stirring
process. This construction is achieved in two steps: from two
independent stirring processes we first obtain the two species SEP
without annihilation then, by a thinning procedure, we construct the
the two species SEP with annihilation.

We start by recalling the graphical construction of the stirring
process.  To each site $x\in \bb Z^d$ attach a copy of the positive
half-axis $\bb R_+$. For each edge $\langle x,y\rangle$ draw a set of
double-arrows sampled according to independent Poisson
point processes with intensity one.
The stirring process $W=\{W_x(t), \, x \in \bb Z^d, \, t\in \bb R_+\}$
is defined as follows: the value $W_x(t) \in \bb Z^d$ is obtained by
placing a marker at time $t=0$ at the point $x$ and letting it evolve
following the path dictated by the arrows.
Given $\zeta\in\{0,1\}^{\bb Z^d}$ the SEP with
initial datum $\zeta$ can be realized as
$\eta_x(t) = \sum_{y\in\bb Z^d} \zeta_y \ind_{\{x\}} (W_y(t))$,
$x\in \bb Z^d$. Let finally $\mb W=\big(W^-, W^+\big)$ be two
independent copies of the stirring process.

The \emph{two species SEP without annihilation} can be described as
follows.  Each site can be: empty, occupied by a particle, occupied by
a anti-particle, or occupied by both a particle and an anti-particle.
The anti-particles evolve according to the stirring process $W^-$
while particles according to $W^+$.  Setting
$\tilde{\mc S}:=\{0,-1,+1,\pm \}^{\bb Z^d}$, the two species SEP
without annihilation is thus the process on the state space $\tilde{\mc S}$
defined by
\begin{equation}
  \label{txw}
  \tilde\xi_x(t)=
  \begin{cases}
    - 1  & \textrm{ if }  \textrm{$\exists y \in \bb Z^d \colon
      \tilde\zeta_y \in\{-1,\pm\}$ and  $W_y^{-} (t) =x$,} \\ 
    & \phantom{\textrm{ if }}
    \textrm{$\nexists z \in \bb Z^d \colon\tilde\zeta_z \in \{1,\pm\}$ and 
         $W_z^{+} (t) =x$,}\\
    + 1  & \textrm{ if }  \textrm{$\nexists y \in \bb Z^d \colon
      \tilde\zeta_y \in\{-1,\pm\}$ and  $W_y^{-} (t) =x$,} \\ 
    & \phantom{\textrm{ if }}
    \textrm{$\exists z \in \bb Z^d \colon\tilde\zeta_z \in\{1,\pm\}$ and 
         $W_z^{+} (t) =x$,}\\
      \pm   & \textrm{ if } \textrm{$\exists y \in \bb Z^d \colon
      \tilde\zeta_y \in\{-1,\pm\}$ and  $W_y^{-} (t) =x$,} \\ 
    & \phantom{\textrm{ if }}
    \textrm{$\exists z \in \bb Z^d \colon\tilde\zeta_z \in\{1,\pm\}$ and 
         $W_z^{+} (t) =x$,}\\    
     0 & \textrm{otherwise.}
  \end{cases}
\end{equation}
Then $\tilde \xi(t)$, $t\ge 0$, is the Markov process
whose generator $\tilde{\ms L}$ acts on local functions
$f\colon \tilde{\ms S}\to \bb R$ as
\begin{equation}
  \label{3.3}
  \begin{split}
     \tilde{\ms L} f (\tilde \xi) =
     &
     \sum_{(x,y)} \Big\{
     \sum_{\alpha\in\{-1,+1\}} \sum_{\beta\in\{0,\pm\}}
     \ind_{\{\alpha\}} (\tilde \xi_x) \ind_{\{\beta\}} (\tilde \xi_y) 
     \big[ f(\tilde\xi^{x,y}) - f(\tilde\xi)\big]
     \\
     &+ \ind_{\{-1\}} (\tilde \xi_x) \ind_{\{+1\}} (\tilde \xi_y) 
     \big[ f(\tilde\xi^{x,y;\pm,0 }) +f(\tilde\xi^{x,y;0,\pm})
     - 2f(\tilde\xi)\big]
     \\
     &+ \ind_{\{0\}} (\tilde \xi_x) \ind_{\{\pm\}} (\tilde \xi_y) 
     \big[ f(\tilde\xi^{x,y;+1, -1 }) + f(\tilde\xi^{x,y;-1, +1 }) 
     - 2f(\tilde\xi)\big]
     \Big\},
   \end{split}
\end{equation}
where the leftmost sum is carried out over the set of \emph{oriented} edges of
$\bb Z^d$, $\tilde\xi^{x,y}$ has been defined in \eqref{1.2} and,
given $\alpha,\beta \in \{0,-1,+1,\pm \}$,
\begin{equation*}
      (\tilde\xi^{x,y;\alpha,\beta})_z:=
      \begin{cases}
        \alpha & \textrm{ if $z=x$,}\\
        \beta & \textrm{ if $z=y$,}\\
        \tilde\xi_z & \textrm{ otherwise.}
      \end{cases}
\end{equation*}
Given a probability $\tilde \wp$ on $\tilde{\ms S}$ we denote by
$\tilde{\bb P}_{\tilde\wp}$ the law of this process with initial
condition $\tilde\wp$ and by $\tilde{\bb E}_{\tilde\wp}$ the
corresponding expectation.

The two species SEP with annihilation $\xi(t)$, $t\ge 0$, can be
finally realized by a \emph{thinning} of two species SEP without
annihilation by recursively removing pairs of particles of different
species that occupy the same site.
This thinning procedure provides a coupling of the processes $\xi(t)$
and $\tilde\xi(t)$  such that for any $t\ge 0$ and $\alpha\in\{-1,1\}$
we have
$\{x\in\bb Z^d \colon \xi_t(x) =\alpha\} \subset \{x\in\bb Z^d \colon
\tilde \xi_t(x) =\alpha\}$ with probability one.

Given $\Lambda\subset\subset \bb Z^d$ and $\alpha\in \{-1,1\}$ we set
$N_{\Lambda,\alpha} (\xi) := \sum_{x\in\Lambda} \ind_{\{\alpha\}}
(\xi_x)$.  Namely, $N_{\Lambda,-1}$ and $N_{\Lambda,1}$ are
respectively the number of anti-particles and the number particles in
$\Lambda$. 
The same notation is used for $\tilde\xi\in \tilde{\ms S}$. 
In the next statement we regard $\wp\in\ms P_\tau(\ms S)$ as a
stationary probability on $\tilde{\ms S}$.

\begin{lemma}
  \label{t:3.3}
  Let $B$ as defined in \eqref{3.1}. Then for each
  $\Lambda \subset\subset \bb Z^d$, $t\ge 0$, and
  $\wp\in \ms P_\tau(\ms S)$ such that
  $\wp(\xi_0=-1)=\wp(\xi_0=1)$,
  \begin{equation*}
    \tilde{\bb E}_\wp \
    \Big( N_{\Lambda,1} (\tilde\xi(t)) - N_{\Lambda,-1} (\tilde\xi(t)) \Big)^2
    \le 2 \, |\Lambda| \,  B (\wp).
  \end{equation*}
\end{lemma}

\begin{proof}
  We write 
  $N_{\Lambda,1} (\tilde\xi(t)) - N_{\Lambda,-1} (\tilde\xi(t)) =
  \sum_{x\in\Lambda}\!
  \big[\ind_{\{+1,\pm\}}\big(\tilde\xi_x(t)\big)
  -\ind_{\{-1,\pm\}} \big(\tilde\xi_x(t)\big)\big]
  $
  and observe that its expectation with respect to  
  $\tilde{\bb P}_\wp$ vanishes.
  Thus
  \begin{equation}
    \label{d+od}
    \begin{split}
      &\tilde{\bb E}_\wp
      \Big(  N_{\Lambda,1} (\tilde\xi(t)) - N_{\Lambda,-1}
      (\tilde\xi(t))
            \Big) ^2
      = \sum_{x\in\Lambda} \sum_{\alpha\in\{-1,1\}}
      \tilde{\bb P}_\wp \big(\tilde\xi_x(t)\in\{\alpha,\pm\} \big)
      \\
      &\qquad\quad
      + \sum_{\substack{x,y \in \Lambda\\ x\neq y}}
      \sum_{\alpha,\beta\in \{-1,1\}} (-1)^{\frac{\alpha+\beta}2 +1}
      \; \tilde{\bb P}_\wp
      \big( \tilde\xi_{x}(t) \in \{\alpha,\pm\} \, ; \,
      \tilde \xi_{y}(t) \in \{\beta,\pm\} \big).  
        \end{split}
  \end{equation}
  Let $p_t(x,y):= \mb{P} \big (W^\alpha_x(t) =y\big)$, i.e.\
  the transition probability of the standard
  continuous time simple symmetric random walk on $\bb Z^d$. By
  \eqref{txw}, the diagonal term in \eqref{d+od}
  is given by
  \begin{equation*}
    \sum_{x\in\Lambda} \sum_{\alpha\in\{-1,1\}}  
    \sum_{y\in \bb Z^d}  \wp (\xi_y =\alpha) p_t(x,y)
    = 2\sigma \, |\Lambda|
  \end{equation*}
  where $\sigma:=\wp(\xi_0=-1)=\wp(\xi_0=1)$.
  
  By \eqref{txw}, for $x\neq y$ and $\alpha,\beta\in\{-1,+1\}$
  \begin{equation*}
    \begin{split}
      &\tilde{\bb P}_\wp \big( \tilde\xi_{x}(t) \in \{\alpha,\pm\} \, ; \,
      \tilde \xi_{y}(t)
      \in \{\beta,\pm\} \big)
      \\
      &\quad 
      =\sum_{x',y'} p^{\alpha,\beta}_t(x',y',x,y)
      \wp (\tilde\xi_{x'} =\alpha
      \,,\,\tilde\xi_{y'}=\beta\big)  
      \\
      &\quad\quad
      - \sum_{x'} p_t(x',x) \wp (\tilde \xi_{x'}=\alpha)\:
      \sum_{y'} p_t(y',y) \wp (\tilde\xi_{y'}=\beta)
    \end{split}
  \end{equation*}
  where
  $p^{\alpha,\beta}_t(x',y',x,y):=\mb P \big( W^\alpha_{x'}(t)
  =x\,,\,W^\beta_{y'}(t)=y \big)$.  Observe that if $\alpha\neq \beta$
  then $p^{\alpha,\beta}_t(x',y',x,y)=p_t(x',x) p_t(y',y)$ while, by
  the Liggett's inequality \cite[Prop.~VIII.1.7]{Li}, for $\alpha=\beta$
  we have $p^{\alpha,\beta}_t(x',y',x,y)\le p_t(x',x) p_t(y',y)$.
  By the invariance of $\wp$ with respect to space shifts
  we deduce that the off diagonal
  term in \eqref{d+od} can be bounded by
  \begin{equation*}
    \begin{split}
    &\sum_{\substack{x,y \in \Lambda\\ x\neq y}}
    \sum_{x',y'} \sum_{\alpha,\beta\in \{-1,1\}}
    p_t(x',x) p_t(y',y)
    \big| \wp \big(\tilde \xi_{0} =\alpha; \tilde \xi_{y'-x'}=\beta\big)\big|
    \\
    &\quad \le
    \sum_{x\in\Lambda} \sum_{z}
      \sum_{x',y'}
      p_t(x',x) p_t(y',x+z)
      \big| \wp \big(\tilde \xi_{0} =\alpha; \tilde \xi_{y'-x'}=\beta\big)\big|
      \\
      &\quad \le
    |\Lambda| \sum_{z}
      \sum_{z'}
      q_t(z',z)
      \big| \wp \big(\tilde \xi_{0} =\alpha; \tilde \xi_{z'}=\beta\big)\big|
      \le |\Lambda| \, B(\wp)
  \end{split}
\end{equation*}
where $q_t(z',z)$ is the transition probability for a rate two
symmetric random walk on $\bb Z^d$, i.e.\ the difference of two
i.i.d.\ rate one symmetric random walks on $\bb Z^d$.
Since
$2\sigma = \sum_{\alpha,\beta\in\{-1,1\}} \big|\wp\big(\xi_0=\alpha
\, ; \,\xi_0=\beta\big) \big|$ the statement follows.
\end{proof}

\section*{Acknowledgments}
We thank L.~Danella and D.~Gabrielli for introducing us to the
Ornstein distance.

\end{document}